%% file: main_file.tex
\documentclass[11pt]{amsart}

\usepackage{mathrsfs}
\usepackage{amsmath, amscd, amsthm,amssymb, amsfonts, verbatim,subfigure}
\usepackage[mathcal]{eucal}

\usepackage[all]{xy}

\usepackage{mathpazo}
\newcommand{\g}{\mathfrak{g}}

\newcommand{\E}{\mscr{E}}

\newcommand{\what}{\widehat}

\newcommand{\til}{\widetilde}
\newcommand{\mscr}{\mathscr}

\newcommand{\br}{\overline}

\newcommand{\iso}{\cong}
\newcommand{\C}{\mathbb C}

\newcommand{\Oo}{\mscr O}
\newcommand{\Z}{\mathbb Z}

\newcommand{\op}{\operatorname}

\newcommand{\mbb}{\mathbb}
\newcommand{\mf}{\mathfrak}
\newcommand{\mc}{\mathcal}

\newcommand{\ip}[1]{\left\langle #1 \right\rangle}
\newcommand{\abs}[1]{\left| #1 \right|}

\newcommand{\R}{\mbb R}
\renewcommand{\d}{\mathrm{d}}
\renewcommand{\epsilon}{\varepsilon}

\newcommand{\dbar}{\br{\partial}}

\DeclareMathOperator{\Diff}{Diff}

\newtheoremstyle{thm}
  {7pt}
  {7pt}
  {\itshape}
  {}
  {\bf}
  {.}
  {5pt}
  {\thmnumber{#2 }\thmname{#1}\thmnote{ (#3)}}

\newtheoremstyle{def}
  {7pt}
  {10pt}
  {\itshape}
  {}
  {\bf}
  {.}
  {5pt}
  {\thmnumber{#2} \thmname{#1}\thmnote{ (#3)}}

\newtheoremstyle{rem}
  {4pt}
  {7pt}
  {}
  {}
  {\itshape}
  {:}
  {3pt}
  {}

\newtheoremstyle{texttheorem}
  {8pt}
  {8pt}
  {\itshape}
  {}
  {\bf}
  {. \hspace{5pt}}
  {3pt}
  {}

\theoremstyle{thm}

\newtheorem*{theorem*}{Theorem}

\newtheorem{theorem}{Theorem}[subsection]
\newtheorem{thm-def}{Theorem/Definition}[theorem]
\newtheorem{proposition}[theorem]{Proposition}

\newtheorem{lemma}[theorem]{Lemma}

\numberwithin{equation}{subsection}

\theoremstyle{def}
\newtheorem{definition}[theorem]{Definition}

\theoremstyle{rem}

\newtheorem*{remark}{Remark}

\theoremstyle{texttheorem}

\newcommand{\cinfty}{C^{\infty}}

\newcommand{\numberedparagraph}{\subsection{}}

\newcommand{\F}{\mathcal{F}}
\newcommand{\G}{\mathcal{G}} 
\newcommand{\EL}{\mathcal{EL}}
\newcommand{\hocolim}{\op{hocolim}}
\theoremstyle{thm}

\begin{document}
\title{A geometric construction of the Witten genus, I}
\author{Kevin Costello}
\begin{abstract}
I describe how the Witten genus of a complex manifold $X$ can be seen from a rigorous analysis of a certain two-dimensional quantum field theory of maps from a surface to $X$. 
\end{abstract}

\maketitle

\input{b_overview.tex}

\input{c_factorization.tex}
\input{d_main_theorem.tex}

\input{e_qft.tex}

\input{f_hcs.tex}


\def\cprime{$'$}

\end{document}

%% file: b_overview.tex
\section{Introduction}
This paper will describe an application of my work on the foundations of quantum field theory (much of it joint with Owen Gwilliam) to topology.  I will show how to construct the Witten genus of a complex manifold $X$ from a rigorous analysis of a quantum field theory of maps from an elliptic curve to $X$. 

Usually the Witten genus is defined by its $q$-expansion.   From this point of view, modularity of the Witten genus is not obvious.      In the construction presented here, however, we find \emph{directly} a function on the moduli space of (suitable decorated) elliptic curves.  It is only after careful calculation that we can compute the $q$-expansion of this function and identify it with the Witten class. 

Hopefully, this construction will give some hints about the mysterious geometric origins of elliptic cohomology. 

I am very grateful to Dennis Gaitsgory, Owen Gwilliam, Mike Hopkins, David Kazhdan, Jacob Lurie, Josh Shadlen, Yuan Shen, Stefan Stolz and Peter Teichner for many helpful conversations about the material in this paper. 

\section{Hochschild homology and the Todd class}
Before turning to elliptic cohomology and the Witten class, I will describe the analog of my construction for the Todd class.

The most familiar way in which the Todd class occurs is, of course, in the Grothendieck-Riemann-Roch theorem.   Let me recall the statement.    Let $X$ be a smooth projective variety, and let $E$ be an algebraic vector bundle on $X$.   Then, the Grothendieck-Riemann-Roch theorem states that 
$$
\sum (-1)^i \op{dim} H^i(X,E) = \int_X \op{Td}(T X) \op{ch}(E).
$$

Another (and closely related) way in which the Todd class appears is in the study of deformation quantization.  There is a rich literature on algebraic and non-commutative analogs of the index theorem: see \cite{Fed96,BreNesTsy99}.   Much of this literature concerns index-type statements on quantizations of general symplectic manifolds.    For the purposes of this paper, we are only interested in the relatively simple case when we are quantizing the cotangent bundle of a complex manifold $X$. 

Let $\Diff_X$ denote the algebra of differential operators on $X$.  Let $\Diff^\hbar_X$ denote the sheaf of algebras on $X$ over the ring $\C[\hbar]$ obtained by forming the Rees algebra of the filtered algebra $\Diff_X$.  Explicitly,
$$
\Diff^\hbar_X \subset \Diff_X \otimes \C[\hbar]
$$
is the subalgebra consisting of those finite sums
$$
\sum \hbar^i D_i
$$
where $D_i$ is a differential operator of order at most $i$. Thus, $\Diff^\hbar_X$ is a $\C[\hbar]$ algebra whose specialization to $\hbar = 0$ is the commutative algebra $\Oo_{T^\ast X}$ of functions on the cotangent bundle of $X$.  When specialized to a non-zero value of $\hbar$,  $\Diff^\hbar_X$ is just $\Diff_X$. 

\numberedparagraph
The theorem we are interested in states that the Todd class of $X$ appears when one compute the Hochschild homology of the algebra $\Diff^\hbar_X$.    The index theorem concerns, ultimately, traces of differential operators.   Since $HH(  \Diff^\hbar_X)$ is the universal recipient of a trace on the algebra $\Diff^\hbar_X$, it is perhaps not so surprising that the Todd class should appear in this context.  

\numberedparagraph
Recall that the Hochschild-Kostant-Rosenberg theorem gives a quasi-isomorphism
$$
\mc{I}_{HKR} : HH ( \Oo_X) \iso \Omega^{-\ast}(X).
$$
Here $HH ( \Oo_X)$ refers to the sheaf of Hochschild chains of $\Oo_X$, and $\Omega^{-\ast}(X)$ refers to the algebra of forms of $X$, with reversed grading.   Applied to the cotangent bundle of $X$, the Hochschild-Kostant-Rosenberg theorem gives an isomorphism
$$
\mc{I}_{HKR} : HH( \Oo_{T^\ast X} ) \iso \Omega^{-\ast} ( T ^\ast X).
$$
The algebra $\Diff^\hbar_X$ is a deformation quantization of $\Oo_{T^\ast X}$.  We will see that the Todd genus appears when we study how $HH( \Oo_{T^\ast X})$ changes when we replace $\Oo_{T^\ast X}$ by $\Diff^\hbar_X$.

\numberedparagraph
Before we state the theorem, we need some notation.  Let $\pi \in  \Gamma ( T^\ast X, \wedge^2 T  ( T^\ast X) )$ denote the canonical Poisson tensor on $T^\ast X$.  Let 
$$
L_\pi : \Omega^i ( T^\ast X) \to \Omega^{i-1} (T^\ast X)
$$
denote the operator of Lie derivative with respect to $\pi$.  Thus, if $i_\pi$ is contraction by $\pi$, 
$$
L_\pi = [ i_\pi, \d_{dR} ].
$$
Note that $L_\pi^2 = 0$, so that $\Omega^{-\ast}(T^\ast X)$ becomes a cochain complex when endowed with differential $L_\pi$.  The cohomology of this complex is called Poisson homology. 

Let
$$
\op{Td}(X) \in H^0( X, \Omega^{-\ast}(X)) = \oplus H^i(X, \Omega^i ) 
$$
be the Todd class of $X$.  Note that the reversal of grading in the de Rham complex means that $\op{Td}(X)$ is an element of cohomological degree $0$. 

The first statement of the theorem is as follows.
\begin{theorem}[Fedosov \cite{Fed96}, Bressler-Nest-Tsyan \cite{BreNesTsy99}]
There is a natural quasi-isomorphism of cochain complexes
$$
HH  ( \Diff^{\hbar}_X ) \simeq \left(  \Omega^{-\ast}  ( T ^\ast X ) [\hbar],  \hbar L_{\pi}  \right )
$$
sending $1 \in HH ( \Diff^{\hbar}_X)$ to 
$$\op{Td}(X) \in \R \Gamma(X, \Omega^{-\ast} (X) ).$$
\end{theorem}

\numberedparagraph
This is a rather weak formulation of the theorem, because both sides in the quasi-isomorphism are simply cochain complexes.  There is a refined version which identifies a certain algebraic structure present on both sides.    It will take a certain amount of preparation to state this refined version.

The operator $L_\pi$ is an order two differential operator with respect to the natural product on $\Omega^{-\ast}(T^\ast X)$. We will let $\{-,-\}_{\pi}$ denote the Poisson bracket on $\Omega^{-\ast}( T^\ast X)$ of cohomological degree $1$ defined by the standard formula
$$
\{a,b\}_{\pi} = L_\pi (a b ) - (L_\pi a ) b - (-1)^{\abs{a}} a L_\pi b. 
$$
The bracket $\{-,-\}_{\pi}$ is of cohomological degree $1$, and satisfies the standard Leibniz rule. Further, $L_\pi$ is a derivation for the bracket $\{-,-\}_{\pi}$. 

\begin{theorem}
There is a quasi-isomorphism of cochain complexes
$$
HH( \Diff^\hbar_X ) \simeq \left(  \Omega^{-\ast}  ( T ^\ast X ) [\hbar],  \hbar L_{\pi}  +  \hbar \{\log \op{Td}(X) , -  \}_\pi \right ).
$$
\end{theorem}
The isomorphism in this theorem is related to that of the previous formulation by conjugating by $\op{Td}(X)$. 

\numberedparagraph
The isomorphism appearing in this second formulation is the one that is compatible with an additional algebraic structure.  The structure is that of an algebra over a certain operad, introduced by Beilinson and Drinfeld \cite{BeiDri04}. 
\begin{definition}
A \emph{Beilinson-Drinfeld algebra} (or BD algebra) is a flat graded $\C[\hbar]$ module $A$ endowed with the following structures.
\begin{enumerate}
\item A commutative unital product.
\item A Poisson bracket $\{-,-\}$ of cohomological degree $1$.
\item A differential $D: A \to A$ of cohomological degree $1$, satisfying $D^2 = 0$ and $D 1 = 0$, such that
$$
D ( a b ) = (D  a ) b + (-1)^{\abs{a}} a (D b) + \hbar \{a,b\}.
$$ 
\end{enumerate}   

A \emph{filtered} Beilinson-Drinfeld algebra is a BD algebra $A$ with a $\C^\times$ action, such that the operator of multiplication by $\hbar$ is of $\C^\times$ weight $1$, the differential and the product on $A$ are of weight $0$, and the Poisson bracket $\{-,-\}$ is of weight $-1$.
\end{definition}
\begin{remark}
The term \emph{filtered} is used because of the well-known isomorphism between filtered vector spaces and $\C^\times$ equivariant vector bundles on $\C$; or equivalently, $\C^\times$-equivariant flat modules over the ring $\C[\hbar]$, where $\hbar$ is given weight $1$.  
\end{remark}
The complex $HH ( \Diff^\hbar_X )$ is endowed with the structure of filtered BD algebra in a natural way. 

The complex $\Omega^{-\ast}( T^\ast X ) [\hbar]$ also has the structure of filtered BD algebra.  The product on $\Omega^{-\ast}( T^\ast X ) [\hbar]$ is the ordinary wedge product of forms. The $\C^\times$ action on $\Omega^{-\ast}(T^\ast X)$ arises from rescaling the fibres of the cotangent bundle $T^\ast X$.      The bracket on $\Omega^{-\ast}(T^\ast X)$ is $\{-,-\}_\pi$, and the differential $\hbar L_\pi +  \hbar \{\log \op{Td}(X) , -  \} $.
\begin{proposition}
The quasi-isomorphism
$$
HH( \Diff^\hbar_X ) \simeq \left(  \Omega^{-\ast}  ( T ^\ast X ) [\hbar],  \hbar L_{\pi}  +  \hbar \{\log \op{Td}(X) , -  \} \right ).
$$
is a quasi-isomorphism of (sheaves of) filtered BD algebras. 
\end{proposition}
This refined statement is enough to fix the Todd class uniquely, as the following lemma shows.
\begin{lemma}
Let 
$$\alpha, \alpha' \in H^0( X, \oplus_{i \ge 1} \Omega^{i} X[i]) = \oplus_{i \ge 1} H^i(X, \Omega^i X) $$ be forms on $X$, with no zero-form component,  and suppose that we have a quasi-isomoprhism sheaves of filtered $BD$ algebras
$$
\left( \Omega^{-\ast}(T^\ast X) [\hbar], \hbar L_\pi + \hbar \{\alpha, - \} \right) \simeq 
\left( \Omega^{-\ast}(T^\ast X) [\hbar], \hbar L_\pi + \hbar \{\alpha', - \} \right).
$$
Then 
$$
[\alpha] = [\alpha'] \in \oplus_{i > 0} H^i(X, \Omega^i(X)) .
$$
\end{lemma}

\numberedparagraph

In this paper I will state a generalization of this characterization of the Todd class, in which the Witten class appears in place of the Todd class.

%% file: c_factorization.tex
\section{Factorization algebras}
Hochschild homology, $K$-theory and the Todd genus are all intimately concerned with the concept of associative algebra.  In order to understand the Witten genus, one needs to consider a richer algebraic structure called a factorization algebra (or more precisely, a translation-invariant factorization algebra on the complex plane $\C$).
   
Factorization algebras can be defined on any smooth manifold: they can be viewed as a ``multiplicative'' analog of a cosheaf.   In the algebro-geometric context,  factorization algebras were first considered by Beilinson and Drinfeld \cite{BeiDri04}. 

In this section,  I will give the formal definition of a factorization algebra, and state a theorem (from \cite{CosGwi10}) which allows one to construct factorization algebras using the machinery of perturbative renormalization developed in \cite{Cos10}.

The approach to constructing factorization algebras developed in \cite{CosGwi10} is a quantum field theoretic analog of the deformation quantization approach to quantum mechanics. Thus, a classical field theory yields a \emph{commutative} factorization algebra (I will define what this means shortly).   Quantizing a classical field theory  amounts to replacing this commutative factorization algebra by a plain factorization algebra.   Just like the Todd genus of a complex manifold $X$ appears when one considers the deformation quantization of the cotangent bundle $T^\ast X$, we will see that the Witten genus arises when we consider the quantization of a commutative factorization algebra associated to a classical field theory whose fields are maps from a Riemann surface to $T^\ast X$.  

\numberedparagraph
The definition of a factorization algebra is rather straightforward to give.   First I will define the notion of prefactorization algebra.
\begin{definition}
Let $M$ be a manifold.  A prefactorization algebra $\F$ on $M$ is consists of the following data.
\begin{enumerate}
\item For every open set $U \subset M$, a cochain complex of topological vector spaces, $\F(U)$.
\item If $U_1, \ldots, U_k$ are disjoint open sets in $M$, all contained in a larger open set  $V$,  a continuous linear map
$$
\F(U_1) \otimes \cdots \otimes \F(U_k) \to \F(V)
$$
(where we use the completed projective tensor product).  These maps are required to be $S_k$-covariant.
\item These maps must satisfy an evident compability condition, which says that different ways of composing these maps yield  the same answer.  

More precisely, suppose that $V_1,\ldots, V_l$ are disjoint open subsets of $V$ such that each $U_i$ is in some $V_j$.  Then the diagram
$$
\xymatrix{  \otimes_{j = 1}^l \left( \otimes_{U_i \subset V_j} \F(U_i) \right) \ar[r] \ar[d]  &  \F(V) \\ 
\otimes_{j = 1}^l \F(V_j) \ar[ur] & } 
$$
is required to commute.
\item $\F(\emptyset) = \C$.
\end{enumerate}
\end{definition}

\subsection{}
A factorization algebra is a prefactorization algebra that satisfies the //locality// axiom. This axiom is the analog of the gluing axiom for sheaves; it expresses how the values on big open sets are determined by the values on small open sets. For sheaves, the gluing axiom says that for any open set $U$ and any cover of that open set, we can determine the value of the sheaf on $U$ from the values on the open cover. For factorization algebras, we require our covers to be fine enough that they capture all the ``multiplicative structure.''

\begin{definition}
Let $U$ be an open set and $\mathfrak{U} := \{ U_i \mid i \in I\}$ a cover of $U$ by open sets.  The cover $\mathfrak{U}$ is \emph{factorizing} if  for any finite collection of points $\{x_1,\ldots,x_k\}$ in $U$, there is a finite collection of pairwise disjoint opens $\{U_{i_1}, \ldots, U_{i_n}\}$ from the cover such that $\{x_1,\ldots,x_k\} \subset U_{i_1} \cup \cdots \cup U_{i_n}$.
\end{definition}

\begin{remark} There is a simple way to find a factorizing cover of a Riemannian manifold $M$.  Namely, fix a Riemannian metric on $M$, and  consider
\[
\{ B_r(x) \,:\, \forall x \in M, \text{ with } 0 < r < InjRad(x)\},
\]
the collection of open balls, running over each point $x \in M$, whose radii are less than the injectivity radius at $x$. Another construction is simply to take the collection of open sets in $M$ diffeomorphic to the open $n$-ball.
\end{remark}

\subsection{}
In order to motivate our definition of factorization algebra, let us write briefly recall the cosheaf axiom. Since, in this paper, we are exclusively interested in factorization algebras and cosheaves with values in cochain complexes, we will discuss the homotopical version of the cosheaf axiom.  

This homotopy cosheaf axiom is defined using the \v{C}ech complex associated to an open cover.   Let $\Phi$ be a pre-cosheaf on $M$, and let $\mf{U} = \{U_i \mid i \in I\}$ be a cover of some open subset $U$ of $M$.  The \v{C}ech complex of $\mf{U}$ with coefficients in $\Phi$ is is defined in the usual way, as
$$
\oplus_k \oplus_{j_1,\ldots, j_k \in I} \Phi (U_{j_1} \cap \cdots \cap U_{j_k}  )   [ k - 1] 
$$
where the differential is defined in the usual way.  We say that $\Phi$ is a homotopy cosheaf if the natural map from the \v{C}ech complex to  $\Phi(U)$ is a quasi-isomorphism, for every open $U \subset M$ and every open cover of $U$.

We will define the notion of factorization algebra in a similar way, except that instead of considering elements $U_i$ of the cover, one considers finite collections of disjoint elements of the cover.

In order to make this precise, we need to introduce some notation.  Let $P I$ denote the set of finite subsets $\alpha \subset I$, with the property that if $j,j' \in \alpha$, $U_j \cap U_{j'} = \emptyset$.

If $\alpha \in PI$, let us define $\F(\alpha)$ by
$$
\F(\alpha) = \otimes_{j \in \alpha} \F( U_j ).
$$
Similarly, if $\alpha_1, \ldots, \alpha_k \in P I$, we will let
$$
\F(\alpha_1, \ldots, \alpha_k ) = \otimes_{j_1 \in \alpha_1,\ldots, j_k \in \alpha_k} \F (U_{j_1} \cap \cdots \cap U_{j_k}).
$$
Note that there are natural maps
$$
p_i : \F( \alpha_1,\ldots, \alpha_k ) \to \F( \alpha_1, \ldots, \what{\alpha_i}, \ldots, \alpha_k ) 
$$
for each $1 \le i \le k$.

This allows us to define the \v{C}ech complex of the cover $\mf{U}$ with coefficients in $\F$ by
$$
\check{C} ( \mf{U}, \F) = \oplus_{k \ge 0} \oplus_{\alpha_1,\ldots, \alpha_k \in P I } \F ( \alpha_1, \ldots, \alpha_k) [k-1].
$$
The differential defined in the usual way.
\begin{definition}
A prefactorization algebra is a factorization algebra if, for every open subset $U \subset M$ and every factorizing cover $\mf{U}$ of $U$, the natural map
$$
\check{C}(\mf{U}, \F) \to \F(U)
$$
is a quasi-isomorphism. 
\end{definition}
We call this axiom the \emph{locality} axiom\footnote{In the version of this paper which appears in the ICM proceedings, a slightly weaker locality axiom was posited.   Owen Gwilliam and I have since found that the stronger axiom presented here satisfies better formal properties.}.     This axiom is, of course, only the correct one for factorization algebras taking values in the category of cochain complexes. 

For factorization algebras taking values in vector spaces, one should require that the map from the \v{C}ech complex to $\F(U)$ is an isomorphism on $H_0$.  This amounts to saying that the sequence
$$
\oplus_{\alpha,\beta} \F(\alpha,\beta) \to \oplus_{\gamma}\F(\gamma) \to \F(U) \to 0
$$
is exact on the right.

  A special case of this locality axiom asserts that,  $U_1, U_2$ are disjoint subsets, then 
$$
\F(U_1 \amalg U_2) = \F(U_1) \otimes \F(U_2).
$$
Similarly, if $\{U_i \in i \in I\}$ is any collection of disjoint open subsets of $M$, then
$$
\F(\amalg U_i) = \otimes_{i \in I} \F(U_i) . 
$$
The cochain complexes $\F(U_i)$ are pointed; the map $\emptyset \to U$ induces a map $\F(\emptyset) = \C \to \F(U)$.  The infinite tensor product of pointed vector spaces is defined as the colimit over all finite tensor products, in the usual way. 

\numberedparagraph
The definition of a factorization algebra is reminiscent of that of an $E_n$ algebra.  In fact, Jacob Lurie has shown the following \cite{Lur09a}.
\begin{proposition}
There is an equivalence of $(\infty,1)$-categories between the category of $E_n$ algebras, and the category of factorization algebras $\mc F$  on $\R^n$ with the additional property that if $B \subset B'$ are balls, the map
$$
\mc F(B) \to \mc F(B')
$$ 
is a quasi-isomorphism. 
\end{proposition}

In another direction, what we call a factorization algebra is the $\cinfty$ analog of a definition introduced by Beilinson and Drinfeld \cite{BeiDri04}.  Beilinson and Drinfeld introduced an algebro-geometric version of the notion of factorization algebra, in order to give a geometric formulation of the axioms of a vertex algebra.  In particular, every vertex algebra yields a factorization algebra. 

\numberedparagraph
As our first example of a factorization algebra, let us see how a differential graded associative algebra $A$ gives rise to a translation-invariant factorization algebra $\mc F_A$ on $\R$.  

We will define the value of $\mc F_A$ on the open intervals of $\R$; the value of $\mc F_A$ on more complicated open subsets is formally determined by this data. 

Let $-\infty \le a < b \le \infty$, and let $(a,b)$ be the corresponding (possibly infinite) open interval in $\R$. We set
$$
\mc F_A((a,b)) = A.
$$
If $(a,b) \subset (c,d)$, then the map 
$$\mc F_A( (a,b) )  \to \mc F_A ((c,d))$$
is the identity map on $A$.

If $-\infty \le a_1 < b_1 < a_2 < b_2 < \dots < a_n < b_n \le \infty$, then the intervals $(a_i,b_i)$ are disjoint.  Part of the data of a factorization algebra is thus a map
$$
\mc F_A ((a_1, b_1 ) )  \otimes \cdots \otimes \mc F_A ((a_n, b_n ) )  \to \mc F_A ( (a_1, b_n ) ). 
$$
Once we identify each $\mc F_A ((a_i, b_i ) )$ with $A$, this map is the $n$-fold product map
\begin{align*}
A ^{\otimes n} & \to A\\
\alpha_1 \otimes  \cdots \otimes \alpha_n & \mapsto \alpha_1 \cdot \alpha_2 \cdot \cdots \cdot \alpha_n.
\end{align*} 
The value of $\F_A$ on any other open subset of $\R$ is determined from this data by the axioms of a factorization algebra.

\section{Descent and factorization homology}
In this paper, we are only interested in translation-invariant factorization algebras on $\C$.    In this section, we will see that associated to such a factorization algebra $\F$, and to an elliptic curve $\E$, equipped with a never-vanishing volume element $\omega$, one can define the \emph{factorization homology} 
$$
FH( E, \F ).
$$ 
Factorization homology is the analog, in the world of factorization algebras, of Hochschild homology.

As motiviation, I will first explain how the Hochschild homology groups of an associative algebra $A$ can be viewed as the factorization homology of the translation-invariant factorization algebra $\F_A$ on $\R$ associated to $A$. 
 
\numberedparagraph 
Factorization algebras satisfy a gluing axiom. Suppose that our manifold $M$ is written as a union $M = U \cup V$ of two open subsets. If $\F$ is a factorization algebra on an open subset $U \subset M$, and if $\G$ is a factorization algebra on $V$, and if 
$$
\phi : \F \mid_{U \cap V} \to \G \mid_{U \cap V}
$$
is an isomorphism of factorization algebras on $U \cap V$, then we can construct a factorization algebra $\mc{H}$ on $M$, whose restriction to $U$ is $\F$ and whose restriction to $V$ is $\G$.

Similarly, factorization algebras satisfy descent.  Suppose that a discrete group $G$ acts properly discontinuously on a manifold $M$, and suppose that $\til{\F}$ is a $G$-equivariant factorization algebra on $M$.  Then, $\til{\F}$ descends to a factorization algebra $\F$ on the quotient $M / G$.

\numberedparagraph
Since we will be using the descent property extensively, it is worth explaining how one constructs the descended factorization algebra $\F$. 

Let us start by mentioning a general construction of factorization algebras from partial data. Let $\mf{U}$ be a basis of open sets of $X$ which is also a factorizing cover (that is, $\mf{U}$ is a \emph{factorizing basis}). 
\begin{definition}
A $\mf{U}$-factorization algebra $\F$ is like a factorization algebra, except that $\F((U)$ is only defined for $U \in \mf{U}$. 
\end{definition}
Let $\F$ be a $\mf{U}$-factorization algebra.  Let us define a prefactorization algebra $\op{Fact} (\F)$ on $X$ by
$$
\op{Fact}(\F)(V) = \check{C}(\mf{U}_V, \mc{F}).
$$
Here $\mf{U}_V$ is the cover of $V$ consisting of those open subsets in the cover $\mf{U}$ which are contained in $V$. 
\begin{lemma}
With this definition, $\op{Fact}(\F)$ is a factorization algebra whose restriction to open sets in the cover $\mf{U}$ is quasi-isomorphic to $\F$. 
\end{lemma}

\subsection{}
Now, let us see how this lemma allows us to construct a descended factorization algebra on $M / G$ from a $G$-equivariant factorization algebra $\til{\F}$  on $M$.   Let $\mf{U}$ be the open cover of $M / G$ consisting of connected open sets which admit a section of the map $M \to M / G$. 

We will define a $\mf{U}$-factorization algebra $\F_0$ by saying that, if $U \in \mf{U}$, 
$$
\F_0(V) = \til{\F} ( \til{U} )
$$
where $\til{U} \subset M$ is a lift of $U$ .  Since $\til{\F}$ is $G$-invariant, this is independent of the choice of lift $\til{U}$.

Now we can apply the extension procedure provided by the lemma. We will let 
$$
\F = \op{Fact}(\F_0).
$$
$\F$ is the desired descended factorization algebra. 

\numberedparagraph
This descent property implies that any translation-invariant factorization algebra $\F$ on $\R$ descends to a factorization algebra $\F^{S^1}$ on $S^1 = \R / \Z$.  We will let
$$
FH ( S^1, \F ) = \F^{S^1} (S^1)
$$
denote the complex of global sections of the factorization algebra $\F^{S^1}$ on $S^1$.  We will refer to the complex $FH ( S^1, \F)$ as the factorization homology complex of $S^1$ with coefficients in $\F$.

\begin{lemma}
Let $\F_A$ denote the factorization algebra on $\R$ associated to a differential graded associated algebra $A$.  Then, there is a natural quasi-isomorphism
$$
FH ( S^1, \F_A  )  \simeq HH(A)
$$
between the factorization homology complex of $S^1$ with coefficients in $\F_A$, and the Hochschild complex of $A$.
\end{lemma}
\begin{proof}
If one analyzes the descent prescription described above, one sees that 
$$
FH ( S^1, \F_A ) = \hocolim_{I_1, \ldots ,I_n} A^{\otimes n}
$$
where the homotopy colimit is over disjoint unordered intervals in $S^1$.  The maps in this homotopy colimit just arise form multiplication in $A$.    One sees that a complex which looks like the ordinary cyclic bar complex emerges from this procedure.  In  \cite{Lur09a} it is proven that the result of this homotopy colimit is indeed homotopy equivalent to the cyclic bar complex. 
\end{proof}

\numberedparagraph
If $\lambda \in \R_{> 0}$, let $S^1_\lambda$ be the quotient of $\R$ by the lattice $\lambda \Z$.  If $\F$ is a translation-invariant factorization algebra on $\R$, then we can descend $\F$ to a factorization algebra on $S^1_\lambda$, and thus define factorization homology $FH (S^1_\lambda, \F)$. When $\lambda = 1$, this coincides with the definition given above. In principle, there is no reason that $FH(S^1_\lambda, \F)$ should be independent of $\lambda$.

If we use the factorization algebra $\F_A$ arising from an associative algebra $A$, then all the factorization homology complexes $FH (S^1_\lambda ,\F_A)$ are canonically isomorphic.  This is because the factorization algebra $\F_A$ on $\R$ is not only translation invariant but also dilation invariant. 

\numberedparagraph

As I mentioned earlier, the factorization algebras relevant to the Witten genus are translation-invariant factorization algebras on $\C$.  Let $\F$ be such a factorization algebra.  Let $E$ be an elliptic curve equipped with a volume element $\omega$.  We will write $E$ as a quotient $\C / \Lambda$ of $\C$ by a lattice $\Lambda$, in such a way that form $\omega$ on $E$ pulls back to the volume form $\d z$ on $\C$. 

Since $\F$ is translation-invariant, it is in particular invariant under $\Lambda$.  Thus, $\F$ descends to a factorization algebra $\F^E$ on $E$.  We define the factorization homology complex of $E$ with coefficients in $\F$ by
$$
FH ( E , \F) = \F^E ( E).
$$
Thus, $FH ( E, \F)$ is the global sections of $\F^E$ on $E$.

Thus, there is an analog of the Hochschild homology groups for every elliptic curve $E$ with volume element $\omega$.  

%% file: d_main_theorem.tex
\section{Main theorem}
The main theorem states that the Witten class of a complex manifold $X$ arises when one considers the factorization homology of a certain sheaf (on $X$) of translation-invariant factorization algebras on $\C$.   Before I state this theorem, I need to recall the definition of the Witten class.

\numberedparagraph
Let $E$ be an elliptic curve, and let $\omega$ be a translation-invariant volume element on $E$.  The Witten class
$$
\op{Wit}(X,E,\omega) \in \R\Gamma(X, \Omega^{-\ast}(X))   = \oplus H^i( X, \Omega^i(X)) 
$$
is a cohomology class, defined as follows.

Let
$$
E_{2k}(E,\omega) = \sum_{\lambda \in \Lambda} \lambda^{-2k}
$$
be the Eisenstein series of the marked elliptic curve $(E,\omega)$. Here, as before, we are writing $E$ as the quotient of $\C$ by a lattice $\Lambda$, in such a way that $\omega$ corresponds to $\d z$.

The Witten class of $X$ is defined by
$$
\op{Wit}(X, E,\omega) = \exp \left\{\sum_{k \ge 2}  \frac{(2k-1)!}{ (2 \pi i )^{2k} }  E_{2k}(E,\omega) ch_{2k} ( T X) \right\}.
$$

If $\tau$ is in the upper half-plane, let $(E_\tau, \omega_\tau)$ denote the elliptic curve associated to the lattice generated by $(1,\tau)$, with volume form $\omega_\tau$ corresponding to $\d z$.  Then, the Witten class has the property that
$$
\lim_{\tau \to i \infty} \op{Wit}(X,E_\tau, \omega_\tau) = e^{-c_1(T_X) / 2} \op{Td}(T X) .
$$
This follows from the identities
\begin{align*}
\lim_{\tau \to i \infty} E_{2k} ( E_\tau, \omega_\tau ) &= 2 \zeta(2k ) \\
\sum_{k \ge 1} 2 \zeta(2k)  \frac{x^{2k}}{ 2k (2 \pi i)^{2k} }  &= \log \left( \frac{x}{1 - e^{-x} } \right) - \frac{x}{2}
\end{align*}
where $\zeta$ is the Riemann zeta function.

\numberedparagraph
Now we can state the theorem.
\begin{theorem}
Let $X$ be a complex manifold, equipped with a trivialization of the second Chern character $ch_2(T X)$.  Then, there is a sheaf $D^{\hbar}_{X,ch}$ of translation-invariant factorization algebras on $\C$, over the algebra $\C[\hbar]$, such that, for every elliptic curve $E$ with volume element $\omega$, there is a natural isomorphism of BD algebras
$$
FH( E, D^{\hbar}_{X,ch} ) \simeq  \left( \Omega^{-\ast} (T^\ast X)[\hbar], \hbar L_\pi + \hbar \{ \log \op{Wit}(X,E,\omega) , - \} \right).
$$

Alternatively, there is an isomorphism of cochain complexes
$$
FH( E, D^{\hbar}_{X,ch} ) \simeq  \left( \Omega^{-\ast} (T^\ast X)[\hbar], \hbar L_\pi\right).
$$
sending
$$
1 \to \op{Wit}(X,E,\omega).
$$
\end{theorem} 

\numberedparagraph

The factorization algebra I construct is  an analytic avatar of the chiral differential operators constructed by Gorbounov, Mailkov and Schechtman \cite{GorMalSch00}.  Note that in their work, the $q$-expansion of the Witten genus appears as the character of the algebra of chiral differential operators. The way the Witten genus appears in this paper is somewhat different, and has the advantage that we see the Witten genus directly as a function on the moduli space of elliptic curves, and not just as a $q$-expansion.   A further analysis of the relationship between the $q$-expansion of the Witten genus and the chiral differential operators has been undertaken by Cheung \cite{Che08}.

%% file: e_qft.tex
\section{Factorization algebras from quantum field theory}
A factorization algebra is the algebraic structure satisfied by the observables of a quantum field theory.  In \cite{CosGwi10} we prove a theorem allowing one to construct factorization algebras using the techniques of perturbative renormalization.  The factorization algebra $D^\hbar_{X,ch}$ encoding the Witten genus will be constructed by quantizing a certain two-dimensional quantum field theory, called holomorphic Chern-Simons theory.
 
Before I discuss this particular quantum field theory,  let me explain, heuristically, why one would expect the observables of a quantum field theory to form a factorization algebra.   Suppose we have a quantum field theory (whatever that is) on a manifold $M$.  Then, for every open subset $U \subset M$, we would expect the set of observables on $U$ -- that is, the set of measurements that can be made by an observer in the open subset $U$ -- to form a vector space, which we call $\mc F(U)$.

If $U \subset V$, then an observable on $U$ will, in particular, be an observable on $V$, so that we get a map $\mc F (U) \to \mc F(V)$.

If $U_1$ and $U_2$ are disjoint, we would expect that all obervables on $U_1 \amalg U_2$ are obtained by taking the product of an observable on $U_1$ with one on $U_2$.  Thus, we would expect that 
$$\mc F (U_1 \amalg U_2) = \mc F (U_1) \otimes \mc F(U_2).$$
Together, these maps give $\mc{F}$ the structure of a factorization algebra.

\numberedparagraph
The idea that the observables of a quantum field theory form a factorization algebra is compatible with two familiar examples.

Quantum mechanics is a quantum field theory on the real line $\R$.   The obserables for quantum mechanics form an associative algebra.  Associative algebras are a particular class of factorization algebras on $\R$. In \cite{CosGwi10}, we show that the factorization algebra associated to the free field theory on $\R$ is, in fact, an $E_1$ algebra; specifically, it is the familiar Weyl algebra of observables of quantum mechanics. 

A second well-understood example is conformal field theory.  The observables of conformal field theory on $\C$ form a vertex algebra; and, as we have seen, vertex algebras are a special class of factorization algebra $\C$.  

\numberedparagraph
Let me now briefly state the results of \cite{CosGwi10} and \cite{Cos10}, allowing one to construct factorization algebras.

In \cite{Cos10}, I gave a definition of a quantum field theory on a manifold $M$, using a synthesis between Wilson's concept of a low-energy effective field theory and the Batalin-Vilkovisky formalism for quantizing gauge theories.   Further, I developed techniques (based on the machinery of perturbative renormalization) allowing one to construct such quantum field theories from Lagrangians.    

Many quantum field theories of physical and mathematical interest, such as Chern-Simons theory and Yang-Mills theory, can be put in this framework.   

The most succinct way to state the main construction of \cite{CosGwi10} is as follows.
\begin{theorem}
Any quantum field theory in the sense of \cite{Cos10}, on a manifold $M$, yields a factorization algebra on $M$.
\end{theorem}
We have seen that factorization algebras satisfy a descent property: if a discrete group $G$ acts properly discontinuously on a manifold $M$, then a $G$-equivariant factorization algebra $\til{\F}$ on $M$ descends to the quotient $M / G$.     Quantum field theories in the sense of \cite{Cos10} satisfy a similar descent property, and the construction of a factorization algebra from a quantum field theory is compatible with descent.  

\section{Deformation quantization in quantum field theory}
In this section, I will explain a little about how one associates a factorization algebra to a classical or quantum field theory.   We will see that the procedure of quantizing a classical field theory can be interpreted in algebraic terms as a kind of deformation quantization in the world of factorization algebras.  

The observables of a classical mechanical system form a commutative algebra, whereas the observables of a quantum mechanical system are only an associative algebra.  We should view this commutativity as being an extra structure present on the observables of a classical system. 

There is a similar story in field theory: the observables of a classical field theory on a manifold $M$ have an extra structure, that of a \emph{commutative} factorization algebra. 

Because we will use this concept several times, it is worth introducing the most general notion. 
\begin{definition}
A \emph{Hopf operad} is an operad in the category of differential graded cocommutative algebras. 
\end{definition}
If $P$ is a Hopf operad, then for each $n$, $P(n)$ is a cocommutative co-dga, and the operad maps are compatible with the coproduct.

If $V,W$ are $P$-algebras, then $V \otimes W$ is also: the coproduct on the spaces $P(n)$ allows one to define the action of $P$ on $V \otimes W$.

For example, the commutative operad $\op{Com}$, defined by $\op{Com}(n) = \C$ for each $n > 0$, is a Hopf operad. The tensor product on commutative algebras defined by this Hopf operad structure is the usual tensor product. 
\begin{definition} Let $P$ be a Hopf operad. Then a $P$ factorization algebra is a factorization algebra $\F$ such that each $\F(U)$ is equipped with the structure of $P$-algebra, and the structure  maps
$$
\F(U_1) \otimes \cdots \F(U_n) \to \F(V)
$$
are maps of $P$-algebras. 
\end{definition}

\numberedparagraph

The main object of interest in a classical field theory is the space of solutions to the Euler-Lagrange equation.  If $U \subset M$ is an open set, let $\EL(U)$ be this space.   Sending $U \mapsto \EL(U)$ defines a sheaf of formal spaces on $M$.  

This sheaf of solutions to the Euler-Lagrange equations can be encoded in the structure of a commutative factorization algebra.    If $U \subset M$ is an open subset, we will let $\Oo(\mc{EL}(U))$ denote the space of functions on $\mc{EL}(U)$. 

Sending $U \mapsto \Oo(\mc{EL}(U))$ defines a commutative factorization algebra: if $U_1,\ldots, U_n \subset U_{n+1}$ are disjoint open subsets, there is a restriction map
$$
\mc{EL}(U_{n+1}) \to \mc{EL}  (U_1) \times \cdots \times \mc{EL} ( U_n) . 
$$
Replacing the map of spaces by the corresponding map of algebras of functions yields the desired structure of commutative factorization algebra.  

\numberedparagraph
In the familiar deformation quantization story, the algebra of observables of a classical mechanical system is a commutative algebra endowed with an extra structure, namely a Poisson bracket. This extra structure is what tells us that the commutative algebra ``wants'' to deform into an associative algebra.

There is a similar picture in the world of factorization algebras: the commutative factorization algebra associated to a classical field theory is endowed with an extra structure, which makes it ``want'' to deform into a plain factorization algebra.

Ordinary Poisson algebras interpolate between commutative algebras and associative (or $E_1$) algebras.  For us, the object describing the observables of a quantum field theory is not an $E_1$ algebra in a symmetric monoidal category; instead, it is an $E_0$ algebra.  An $E_0$ algebra in vector spaces is simply a vector space with an element.   An $E_0$ algebra in any symmetric monoidal category is an object of this category with a map from the unit object.    An $E_0$ algebra in the symmetric monoidal category of factorization algebras is simply a factorization algebra, as every factorization algebra is equipped with a unit. 
 
Thus, the analog of the Poisson operad  we are searching for is an operad that interpolates between the commutative operad and the $E_0$ operad.  Such an operad was constructed by Beilinson and Drinfeld \cite{BeiDri04}; we will call it the BD operad\footnote{Beilinson and Drinfeld called this operad the Batalin-Vilkovisky operad.  However, in the literature, the Batalin-Vilkovisky operad has, unfortunately, come to refer to a different object.}.  In the introduction I described what an algebra over the BD operad is.  Here is a description of the operad itself.
 \begin{definition}
Let $P_0$ be the graded operad over $\C$  generated by a commutative and associatve product, $\ast$, and a Poisson bracket $\{-,-\}$ of cohomological degree $+1$.

Let $BD$ denote the differential graded  operad over the ring $\C[\hbar]$ which, as a graded operad, is simply $P_0 \otimes \C[\hbar]$, but which is equipped with differential 
$$
\d \ast = \hbar \{-,-\}.
$$
\end{definition}
If we specialize to $\hbar =0$, we find the $BD$ operad becomes the operad $P_0$. If we specialize to $\hbar = 1$, however, the $BD$ operad becomes the $E_0$ operad.   Thus, we find that the operad $P_0$ bears the same relationship to the operad $E_0$ as the usual Poisson operad bears to the associative operad $E_1$. 

The operads $P_0$ and $BD$ are both Hopf operads. (Recall that this means that the tensor product of two $P_0$ (or $BD$) algebras has a natural structure of $P_0$ (or $BD$) algebra). 

Thus, we can talk about $P_0$ factorization algebras and $BD$ factorization algebras.  A $P_0$ factorization algebra $\F_{cl}$ is a factorization algebra such that, for each $U$ in $M$, $\F^{cl}(U)$ is a $P_0$ algebra, and all the structure maps are maps of $P_0$ algebras; and similar for $BD$ factorization algebras.
\begin{definition}
 A  \emph{quantization} of a $P_0$ factorization algebra $\F_{cl}$ is a BD factorization algebra factorization algebra $\F_q$,
$$
\F_q \otimes_{\C[\hbar]} \C \iso \F_{cl}
$$
of Poisson factorization algebras. 
\end{definition}
Thus, $\F_q(U)$ is a quantization of the $P_0$ algebra $\F_{cl}(U)$, for each $U \subset M$. 
\numberedparagraph
Now we can restate the main results of \cite{CosGwi10}.
\begin{theorem}
Every classical field theory on $M$ gives rise to a Poisson factorization algebra on $M$.   A quantization of this classical field theory (in the sense of \cite{Cos10}) gives rise to a quantization of this Poisson factorization algebra. 
\end{theorem}
What I mean by a classical field theory on $M$ is detailed in \cite{Cos10}, but it is something rather familiar.  There is a space of fields, which is taken to be the space $\E$ of sections of some vector bundle $E$ on $M$, or more generally some space of maps $M \to N$ to some other manifold $N$.  In addition, there is an action functional $S : \E \to \R$ (or to $\C$), which is taken to be the integral of some Lagrangian density.    When dealing with theories with gauge symmetry, this basic picture needs to be modified by the introduction of fields which possess a cohomological degree.  This more sophisticated picture is known as the Batalin-Vilkovisky formalism.

In \cite{Cos10, CosGwi10} we always work in the Batalin-Vilkovisky formalism. Thus, our space of classical fields is equipped with a symplectic form of cohomological degree $-1$.  The fact that the factorization algebra of observables of a classical field theory is equipped with a Poisson bracket of degree $+1$ is simply a version of the familiar statement that the algebra of functions on a symplectic manifold has a natural Poisson bracket.

%% file: f_hcs.tex
\section{Holomorphic Chern-Simons theory}
As we have seen,   when we work in the BV formalism, the space of classical fields is a (typically infinite dimensional) differential graded manifold equipped with a symplectic form of cohomological degree $-1$.  The action functional is a secondary object in this approach.  The differential on the space of fields preserves the symplectic form, and thus, at least locally, is given by Poisson bracket with some Hamiltonian function $S$, of cohomological degree zero. This function is the classical action. 

In the paper \cite{AleKonSch95}, Alexandrov, Kontsevich, Schwartz and Zabronovsky introduced a beautiful and general method for constructing classical field theories in the BV formalism.  Many quantum field theories studied in mathematics arise from the AKSZ construction.  For example, Chern-Simons theory,  Rozansky-Witten theory and the Poisson $\sigma$ model all fit very naturally into this framework.

For us, the relevance of the AKSZ construction is that the classical field theory related to the Witten genus arises most naturally from the AKSZ construction. 

Before I introduce the AKSZ construction, we need some notation.
\begin{definition}
A \emph{differential graded manifold} is a smooth manifold $X$ equipped with a sheaf $\Oo_X$ of differential graded commutative algebras over $\C$, with the property that $\Oo_X$ is locally isomorphic as a graded algebra to $\cinfty_X [[x_1,\ldots, x_n]]$, where $x_i$ are formal variables of cohomological degree $d_i \in \Z$. 
\end{definition}
In this definition, $\cinfty_X$ refers to the sheaf of complex-valued smooth functions on $X$.  One can talk about geometric structures -- such as Poisson or symplectic structures -- on a differential graded manifolds.   

If $X$ is a smooth manifold, we will let $X_{dR}$ denote the dg manifold whose underlying smooth manifold is $X$, and whose sheaf of functions is the complexified de Rham complex $\Omega^\ast_X$ of $X$. If $X$ is a complex manifold, we will let $X_{\dbar}$ denote the dg manifold whose underlying smooth manifold is $X$, and whose sheaf of functions is the Dolbeault complex $\Omega^{0,\ast}_X$. 

\numberedparagraph
Now we can explain the AKSZ construction.  Suppose we have a compact differential graded manifold $M$, equipped with volume element of cohomological degree $k$. Let $X$ be a differential graded manifold with a symplectic form of cohomological degree $l$.  Then, the infinite-dimensional differential graded manifold $\op{Maps}(M,X)$ acquires a symplectic form of cohomological degree $l - k$.

If $f : M \to X$ is a map, then the tangent space to $\op{Maps}(M,X)$ at $f$ is
$$
T_f \op{Maps}(M,X) = \Gamma(M, f^\ast T X).
$$
We define a pairing on $T_f \op{Maps}(M,X)$ by the formula
$$
\ip{\alpha,\beta} = \int_M \ip{\alpha,\beta}_X.
$$
Since the integration map $\int_M : \cinfty(M) \to \R$ is of cohomological degree $-k$, and the symplectic pairing on $T X$ is of cohomological degree $m$, the pairing on $T_f \op{Maps}(M,X)$ is of cohomological degree $m - k$.  The case of interest in the Batalin-Vilkovisky formalism is when $m-k = -1$.  

There is a variation of this construction which applies when the source manifold $M$ is non-compact.  In this situation, the space $\op{Maps}(M,X)$ has a natural integrable distribution given by the subspace
$$
\Gamma_c(M, f^\ast T X) \subset  T_f^c \op{Maps}(M,X) \subset T_f \op{Maps}(M,X)
$$
consisting of compactly supported tangent vector fields.    In this situation,  instead of having a symplectic pairing on $T_f \op{Maps}(M,X)$, we only have one on the distribution $T_f^c\op{Maps}(M,X)$.    The action functional, instead of being a closed one-form on $\op{Maps}(M,X)$, is a closed one-form on the leaves of the foliation.

\numberedparagraph

There are two broad classes of AKSZ theories which are commonly considered.  These are the theories of Chern-Simons type, and the theories of holomorphic Chern-Simons type. 

The two classes of theories are distinguished by the nature of the source dg manifold.  In theories of Chern-Simons type, the source differential graded ringed space is $M_{dR}$, where $M$ is an oriented manifold.   The orientation on $M$ gives rise to a volume element on $M_{dR}$ of cohomological degree $\op{dim}(M)$. 

The target manifolds for Chern-Simons theories of dimension $k$ are dg symplectic manifolds of dimension $k-1$.   For example, perturbative Chern-Simons theory arises when we take the target to be the dg manifold whose underlying manifold is a point, and whose algebra of functions is the algebra $C^\ast(\mf g)$ of cochains on a semi-simple Lie algebra $\mf g$.  The Killing form endows this dg manifold with a symplectic form of cohomological degree $2$.   This theory is perturbative, because maps $M_{dR} \to C^\ast(\g)$ are the same as connections on the trivial principal $G$ bundle which are infinitesimally close to the trivial connection.

Non-perturbative Chern-Simons theory arises from a genearlized form of the AKSZ construction whcih takes the stack $BG$ as the target manifold.  Vector bundles on $B G$ are the same as $G$-modules; the tangent bundle of $B G$ is the adjoint module $\g[1]$.   The Killing form on $\g$ is $G$-equivariant, and so gives rise to a symplectic form on $B G$ of cohomological degree $2$.

Rozansky-Witten theory also arises from this framework.  Let $X_{\dbar}$ be a holomorphic symplectic manifold.  Let us work over the base ring $\C[q,q^{-1}]$ where $q$ is a parameter of degree $-2$.  Then the symplectic form $q^{-1} \omega$ on $X_{\dbar}$ is of cohomological degree $2$, and so we can define a $3$-dimensional Chern-Simons type theory.    The fields of this theory are maps $M_{dR} \to X_{\dbar}$, where $M$ is a $3$-manifold, and everything takes place over the base ring $\C[q,q^{-1}]$. 

Another example is the Poisson $\sigma$-model of \cite{Kon97, CatFel01}.  Here, the source is $\Sigma_{dR}$ where $\Sigma$ is a smooth surface. The target is the differential graded manifold $T^\ast[1] X$, whose underlying smooth manifold is $X$, and whose algebra of functions is $\Gamma(X, \wedge^\ast T X)$.    The Schouten-Nijenhuis bracket $\{-,-\}$ endows this dg manifold with a symplectic form of cohomological degree $1$.  The differential on $\Gamma(X,\wedge^\ast TX)$ is given by bracketing with the Poisson tensor $\pi$.

\numberedparagraph
Let us now discuss holomorphic Chern-Simons theory, which is the only quantum field theory we will be concerned with in this paper.     In holomorphic Chern-Simons theory, the source dg manifold is $M_{\dbar}$, where $M$ is a complex manifold equipped with a never-vanishing holomorphic volume element $\omega$ (thus, $M$ is a Calabi-Yau manifold).  This volume form can be thought of as a volume element on $M_{\dbar}$ of cohomological degree $\op{dim}_{\C}(M)$.    Integration against this volume element is simply the map
\begin{align*}
\Omega^{0,\op{dim}_{\C}(M)} (M) &\to \C \\
\alpha &\mapsto \int_M \omega \wedge \alpha.
\end{align*}

Theories of holomorphic Chern-Simons type on Calabi-Yau manifolds of complex dimension $k$ can thus be constructed from dg symplectic manifolds with a symplectic form of cohomological degree $k-1$.  

In this paper, we are only interested in one-dimensional holomorphic Chern-Simons theories.  In these theories, the source dg manifold is $\Sigma_{\dbar}$, where $\Sigma$ is a Riemann surface equipped with a never-vanishing holomorphic volume form. The target is $X_{\dbar}$, where $X$ is a holomorphic symplectic manifold.  (The holomorphic symplectic form can be thought of as a dg symplectic form on $X_{\dbar}$).  

\numberedparagraph
We can now give a more precise statement of the theorem relating elliptic cohomology and the Witten genus.

\begin{theorem}
Let $X$ be a complex manifold.  Then,
\begin{enumerate}
\item The obstruction to quantizing the holomorphic Chern-Simons theory whose fields are maps $\C_{\dbar} \to (T^\ast X)_{\dbar}$ is 
$$\op{ch}_2 ( T X ) \in H^2(X, \Omega^2_{cl}(X))$$
where $\Omega^2_{cl}(X)$ is the sheaf of closed holomorphic $2$-forms on $X$.
\item If this obstruction vanishes (or, more precisely, is trivialized), then we can quantize holomorphic Chern-Simons theory to yield a factorization algebra on $\C$ with values in quasi-coherent sheaves on $X_{dR}\times \op{Spec} \C[\hbar]$.  We will call this factorization $D^{\hbar}_{X,ch}$.
\item If $E$ is an elliptic curve, then there is a quasi-isomorphism of $BD$ algebras in quasi-coherent sheaves on $X_{dR} \times \C[\hbar]$ 
$$
FH (  E,  D^{\hbar}_{X,ch } ) \simeq \left(  \Omega^{-\ast} ( T^\ast X ) [\hbar] , \hbar L_\pi + \hbar \{ \log \op{Wit}(X,E) - \} \right) .
$$
\end{enumerate}
\end{theorem}

\numberedparagraph
Recall that the factorization homology complex $FH ( E, D^{\hbar}_{X,ch} )$ is defined  by first constructing a factorization algebra $D^{E,\hbar}_{X,ch}$ on $E$, using the descent property of factorization algebras; and then taking global sections.

Quantum field theories in the sense of \cite{Cos10} have a descent property similar to that satisfied by factorization algebras, and the construction of a factorization algebra from a quantum field theory is compatible with descent.    The quantum field theory on an elliptic curve $E$ which arises by descent from holomorphic Chern-Simons theory on $\C$ is simply holomorphic Chern-Simons theory on $E$.

Thus, one can interpret the factorization homology group $FH ( E, D^{\hbar}_{X,ch} )$ in terms of holomorphic Chern-Simons theory on the elliptic curve $E$.   From this point of view, $FH( E, D^{\hbar} _{X,ch} ) $ is the cochain complex of global observables for the holomorphic Chern-Simons theory of maps $E \to T^\ast X$.  

This theorem is proved using the Wilsonian approach to quantum field theory developed in \cite{Cos10}.  The result is then translated into the language of factorization algebras.    The proof appears in \cite{Cos10b}.

\numberedparagraph
As I mentioned earlier, this factorization algebra is an analytic avatar of the chiral algebra of chiral differential operators constructed by Gorbounov, Malikov and Schechtman \cite{GorMalSch00}.  A detailed description of the factorization algebra constructed from holomorphic Chern-Simons theory, and its relationship to chiral differential operators, will appear elsewhere.    

In the physics literature, Kapustin\cite{Kap05} and Witten \cite{Wit05} have argued that chiral differential operators arise from what is called, in section 3.3 of \cite{Wit05}, a non-linear $\beta$-$\gamma$ system.     This non-linear $\beta$-$\gamma$ system is the same as what I have called holomorphic Chern-Simons theory.

%% file: main_file.bbl
\begin{thebibliography}{AKSZ97}

\bibitem[AKSZ97]{AleKonSch95}
M.~Alexandrov, M.~Kontsevich, A.~Schwarz and O.~Zabronovsky, \textsl{ The
  Geometry of the master equation and topological field theory},
\newblock Internat. J. Modern Phys. \textbf{ 12}(7), 1405--1429 (1997),
  {hep-th/9502010}.

\bibitem[BD04]{BeiDri04}
A.~Beilinson and V.~Drinfeld,
\newblock \textsl{ Chiral algebras}, volume~51 of \textsl{ American
  Mathematical Society Colloquium Publications},
\newblock American Mathematical Society, Providence, RI, 2004.

\bibitem[BNT02]{BreNesTsy99}
P.~Bressler, R.~Nest and B.~Tsyagn, \textsl{ Riemann-{R}och theorems via
  deformation quantization, I},
\newblock Adv. Math. \textbf{ 167}(1), 1--25 (2002), {math.AG/9904121}.

\bibitem[CF01]{CatFel01}
A.~Cattaneo and G.~Felder, \textsl{ Poisson sigma-models and deformation
  quantization},
\newblock (2001), {hep-th/0102208}.

\bibitem[Che08]{Che08}
P.~Cheung, \textsl{ The {W}itten genus and vertex algebras}, \newblock(2008), {arXiv:0811.1418}.

\bibitem[CG10]{CosGwi10}
K.~Costello and O.~Gwilliam, \textsl{ Factorization algebras in perturbative
  quantum field theory}, in progress
\newblock (2010).

\bibitem[Cos10b]{Cos10}
K.~Costello, \textsl{ Renormalization and effective field theory},
\newblock (2010). { \\Available at \verb=http://www.math.northwestern.edu/~costello/=}.

\bibitem[Cos10a]{Cos10b}
K.~Costello, \textsl{ A geometric construction of the {W}itten genus, II},
\newblock (2010). { \\Available at \verb=http://www.math.northwestern.edu/~costello/=}.


\bibitem[Fed96]{Fed96}
B.~Fedosov,
\newblock \textsl{ Deformation quantization and index theory},
\newblock Akademie Verlag, 1996.

\bibitem[GMS00]{GorMalSch00}
V.~Gorbounov, F.~Malikov and V.~Schechtman, \textsl{ Gerbes of chiral
  differential operators},
\newblock Math. Res. Lett. \textbf{ 7}(1), 55--66 (2000).

\bibitem[Kap05]{Kap05}
A.~Kapustin, \textsl{ Chiral de Rham complex and the half-twisted sigma-model},
\newblock (2005), {hep-th/0504074}.

\bibitem[Kon03]{Kon97}
M.~Kontsevich, \textsl{ Deformation quantization of {P}oisson manifolds},
\newblock Lett. Math. Phys. \textbf{ 66}(3), 157--216 (2003), {q-alg/0709040}.

\bibitem[Lur09]{Lur09a}
J.~Lurie, \textsl{ Derived algebraic geometry {VI}: $E_k$ algebras.},
\newblock (2009), {\\ \verb=http://math.mit.edu/~lurie/papers/DAG-VI.pdf=}.

\bibitem[Wit05]{Wit05}
E.~Witten, \textsl{ Two-dimensional models with (0,2) supersymmetry:
  perturbative aspects},
\newblock (2005), {hep-th/0504078}.

\end{thebibliography}
